\documentclass[english]{smfart}

\usepackage[english]{babel}
\usepackage{smfthm}
\usepackage{bull}
\usepackage{mathrsfs}
\usepackage{amssymb,bm}
\usepackage{graphicx}
\usepackage{hyperref}
\usepackage{smfhyperref}
\usepackage[all]{xy}

\renewcommand{\bar}[1]{\overline{#1}}
\renewcommand{\ker}{\textnormal{ker}}
\newcommand{\im}{\textnormal{im}}
\newcommand{\coker}{\textnormal{coker}}
\newcommand{\A}{\mathbb{A}}
\newcommand{\C}{\mathbb{C}}
\newcommand{\R}{\mathbb{R}}
\newcommand{\rk}{\textnormal{rk} \,}
\newcommand{\prim}{{\mathrm{P} \hspace{-0.3mm}}}
\DeclareMathOperator{\DR}{DR}

\NumberTheoremsAs{equation}

\author{Nicolas Martin}
\address{Centre de math\'ematiques Laurent Schwartz, \'Ecole polytechnique, Universit\'e Paris-Saclay,
F-91128 Palaiseau cedex, France}
\email{nicolas.martin@polytechnique.edu}
\title{Behaviour of some Hodge invariants by middle convolution}
\alttitle{Comportement d'invariants de Hodge par convolution interm\'ediaire}

\begin{document}

\frontmatter

\begin{abstract}
Following a paper of Dettweiler and Sabbah, this article studies the behaviour of various Hodge invariants by middle additive convolution with a Kummer module. The main result gives the behaviour of Hodge numerical data at infinity. We also give expressions for Hodge numbers and degrees of some Hodge bundles without making the hypothesis of scalar monodromy at infinity, which generalizes the results of Dettweiler and Sabbah.
\end{abstract}

\begin{altabstract}
\noindent
--- Suivant les travaux de Dettweiler et Sabbah, cet article s'int\'eresse au comportement d'invariants de Hodge par convolution interm\'ediaire additive par un module de Kummer. Le r\'esultat principal pr\'ecise le comportement de donn\'ees num\'eriques de Hodge \`a l'infini. Nous explicitons \'egalement le comportement des nombres de Hodge et des degr\'es de certains fibr\'es de Hodge sans faire l'hypoth\`ese de monodromie scalaire \`a l'infini, g\'en\'eralisant ainsi les r\'esultats de Dettweiler et Sabbah.
\end{altabstract}

\subjclass{14D07, 32G20, 32S40}
\keywords{D-modules, middle convolution, rigid local system, Katz algorithm, Hodge theory}
\maketitle
\mainmatter

The initial motivation to study the behaviour of various Hodge invariants by middle additive convolution is Katz's algorithm \cite{Kat96}, which makes it possible to reduce a rigid irreducible local system $\mathscr{L}$ on a punctured projective line to a rank-one local system. This algorithm is a successive application of tensor products with a rank-one local system and middle additive convolutions with a Kummer local system, and terminates with a rank-one local system. We assume that
the monodromy at infinity of $\mathscr{L}$ is scalar, so this property is preserved throughout the algorithm.

\newpage

If we assume that the eigenvalues of the local monodromies of $\mathscr{L}$ have absolute value one, such a local system underlies a variation of polarized complex Hodge structure unique up to a shift of the Hodge filtration \cite{Sim90,Del87}, and this property is preserved at each step of Katz's algorithm. The work of Dettweiler and Sabbah \cite{DS13} is devoted to computing the behaviour of Hodge invariants at each step of the algorithm.

\medskip

Our purpose in this article is to complement the previous work of Dett\-weiler and Sabbah without assuming that the monodromy at infinity is scalar, and to do that, we take up the notation introduced in \cite[\S2.2]{DS13} and recalled in {\S}\ref{11}. More precisely, our main result consists in making explicit the behaviour of the nearby cycle local Hodge numerical data at infinity by middle additive convolution with the Kummer module $\mathcal{K}_{\lambda_0}$. Considering a regular holonomic $\mathscr{D}_{\A^1}$-module $M$ verifying various assumptions, whose singularities at finite distance belong to $\bm{x}=\{x_1,\hdots,x_r\}$, we denote by $\textnormal{MC}_{\lambda_0}(M)=M \: {*}_{\textnormal{mid}} \: \mathcal{K}_{\lambda_0}$ this convolution and show the following theorem (see {\S}\ref{11} for the notation and assumptions). In the following, we set $\gamma_0 \in (0,1)$ such that $\exp(-2i\pi\gamma_0)=\lambda_0 \neq 1$.

\medskip

\begin{theo}\label{theo1}
Let $\mathscr{M}^{\textnormal{min}}$ be the $\mathscr{D}_{\mathbb{P}^1}$-module minimal extension of $M$ at infinity. Given $\gamma \in [0,1)$ and $\lambda=\exp(-2i\pi\gamma)$, we have:
$$\nu_{\infty,\lambda,\ell}^p(\textnormal{MC}_{\lambda_0}(M))=\left\{
\begin{array}{cl}
  \nu_{\infty,\lambda\lambda_0,\ell}^{p-1}(M) \quad &\textnormal{if } \gamma \in (0,1-\gamma_0) \\[1mm]
  \nu_{\infty,\lambda\lambda_0,\ell}^p(M) \quad &\textnormal{if } \gamma \in (1-\gamma_0,1) \\[1mm]
	\nu_{\infty,\lambda_0,\ell+1}^p(M) \quad &\textnormal{if } \lambda=1 \\[1mm]
  \nu_{\infty,1,\ell-1}^{p-1}(M) \quad &\textnormal{if } \lambda=\bar{\lambda_0}, \ \ell \geq 1 \\[1mm]
	h^p H^1(\mathbb{P}^1,\DR\mathscr{M}^{\textnormal{min}}) \quad &\textnormal{if } \lambda=\bar{\lambda_0}, \ \ell = 0.
\end{array}
\right.$$
\end{theo}

\medskip

This result has applications beyond Katz's algorithm since it enables us to give another proof of a theorem of Fedorov \cite{Fed18}, which completely determines the Hodge numbers of the variations of Hodge structures corresponding to hypergeometric differential equations; this work is developed in \cite{Mar21}.

\medskip

In addition, we get general expressions for Hodge numbers $h^p$ of the variation and degrees $\delta^p$ of some Hodge bundles (recalled in {\S}\ref{11}) which generalize those of Dettweiler and Sabbah. The results are the following.

\medskip

\begin{theo}\label{theo2}
The local invariants $h^p(\textnormal{MC}_{\lambda_0}(M))$ are given by:
\begin{multline*}
h^p(\textnormal{MC}_{\lambda_0}(M))= \\
\sum_{\gamma \in [0,\gamma_0)} \nu_{\infty,\lambda}^{p}(M) + \sum_{\gamma \in [\gamma_0,1)} \nu_{\infty,\lambda}^{p-1}(M) + h^p H^1(\mathbb{A}^1,\DR M) - \nu_{\infty,\lambda_0,\textnormal{prim}}^{p-1}(M).
\end{multline*}
\end{theo}

\medskip

\begin{theo}\label{theo3}
The global invariants $\delta^p(\textnormal{MC}_{\lambda_0}(M))$ are given by:
\begin{multline*}
\delta^p(\textnormal{MC}_{\lambda_0}(M)) = \\
\delta^p(M) + \sum_{\gamma \in [\gamma_0,1)} \nu_{\infty,\lambda}^{p}(M) - \sum_{i=1}^r \bigg( \mu_{x_i,1}^p(M) + \sum_{\gamma \in (0,1-\gamma_0)} \mu_{x_i,\lambda}^{p-1}(M) \bigg).
\end{multline*}
\end{theo}

\medskip

\section{Hodge numerical data and modules of normal crossing type}

\numberwithin{equation}{section}

\vspace{0.1cm}

\subsection{Hodge invariants}

\label{11}

\medskip

In this section, we recall the definition of local and global invariants introduced in \cite[{\S}2.2]{DS13} (all references to \cite{DS13} are made to the published paper). Let $\Delta$ be a disc centered at 0 with coordinate $t$ and let $(V,F^{\bullet}V,\nabla)$ be a variation of polarizable Hodge structure on $\Delta^*=\Delta \setminus \{0\}$ of weight $0$. We denote by $M$ the corresponding $\mathscr{D}_\Delta$-module minimal extension at $0$.

\bigskip

\noindent
\textbf{Nearby cycles.} For $a \in (-1,0]$ and $\lambda=e^{-2i\pi a}$, the nearby cycle space at the origin $\psi_\lambda(M)$ is equipped with the nilpotent endomorphism $\mathrm{N}=-2i\pi(t \partial_t-a)$ and the Hodge filtration is such that $\mathrm{N} F^p\psi_\lambda(M) \subset F^{p-1}\psi_\lambda(M)$. The mono\-dromy filtration induced by $\mathrm{N}$ enables us to define the spaces $\prim_\ell \psi_\lambda(M)$ of primitive vectors, equipped with a polarizable Hodge structure (see \cite[{\S}3.1.a]{SS19} for more details). The nearby cycle local Hodge numerical data are defined by

\vspace{-0.1cm}

$$\nu_{\lambda,\ell}^p(M):=h^p(\prim_\ell \psi_\lambda(M))=\dim \textnormal{gr}_F^p \prim_\ell \psi_\lambda(M),$$

\medskip

\noindent
with the relation $\nu_{\lambda}^p(M):=h^p \psi_\lambda(M)=\sum\limits_{\ell \geq 0}\sum\limits_{k=0}^\ell \nu_{\lambda,\ell}^{p+k}(M)$. We set

\medskip

$$\nu_{\lambda,\textnormal{prim}}^p(M):=\sum\limits_{\ell \geq 0} \nu_{\lambda,\ell}^p(M) \ \textnormal{ and } \ \nu_{\lambda,\textnormal{coprim}}^p(M):=\sum\limits_{\ell \geq 0} \nu_{\lambda,\ell}^{p+\ell}(M).$$

\bigskip

\noindent
\textbf{Vanishing cycles.} For $\lambda \neq 1$, the vanishing cycle space at the origin is given by $\phi_\lambda(M)=\psi_\lambda(M)$ and comes with $\textrm{N}$ and $F^p$ as before. For $\lambda=1$, the Hodge filtration on $\phi_1(M)$ is such that $F^p \prim_\ell \phi_1(M)=\mathrm{N}(F^p \prim_{\ell+1} \psi_1(M))$. Similarly to nearby cycles, the vanishing cycle local Hodge numerical data is defined by

\vspace{-0.1cm}

$$\mu_{\lambda,\ell}^p(M):=h^p(\prim_\ell \phi_\lambda(M))=\dim \textnormal{gr}_F^p \prim_\ell \phi_\lambda(M).$$

\bigskip

\noindent
\textbf{Degrees $\delta^p$.} For a variation of polarizable Hodge structure $(V,F^\bullet V,\nabla)$ on $\A^1 \setminus \bm{x}$, we denote by $M$ the underlying $\mathscr{D}_{\A^1}$-module minimal extension at each point of $\bm{x}$. The Deligne extension $V^0$ of $(V,\nabla)$ on $\mathbb{P}^1$ is contained in $M$, and we set

\vspace{-0.5cm}

$$\delta^p(M)=\deg \textnormal{gr}_F^p V^0.$$

\medskip

In this paper, we are mostly interested in the behaviour of the nearby cycle local Hodge numerical data at infinity by middle convolution with the Kummer module $\mathcal{K}_{\lambda_0}=\mathscr{D}_{\A^1}/\mathscr{D}_{\A^1} \cdot (t\partial_t-\gamma_0)$, with $\gamma_0 \in (0,1)$ such that $\exp(-2i\pi\gamma_0)=\lambda_0$. This operation is denoted by $\textnormal{MC}_{\lambda_0}$. Note that $\mathcal{K}_{\lambda_0}$ is equipped with the trivial Hodge filtration with jump at zero : $F^p \mathcal{K}_{\lambda_0}=\mathcal{K}_{\lambda_0}$ for $p \leq 0$ and $F^p \mathcal{K}_{\lambda_0}=0$ for $p \geq 1$.

\bigskip

\noindent
\textbf{Assumptions.} As in \cite[Assumption 1.2.2(1)]{DS13}, we assume in what follows that $M$ is an irredu\-cible regular holonomic $\mathscr{D}_{\A^1}$-module, not isomorphic to $(\C[t],\textnormal{d})$ and not supported on a point.

\medskip

\subsection{Modules of normal crossing type}

Let us consider $X$ a polydisc in $\C^n$ with analytic coordinates $x_1,...,x_n$, $D$ the divisor $\{x_1 \cdots x_n=0\}$ and $M$ a coherent $\mathscr{D}_X$-module of normal crossing type (notion defined in \cite[{\S}3.2]{Sai90}). For every $\alpha =(\alpha_1,...,\alpha_n) \in \R^n$, we define the sub-object $M_\alpha=\bigcap_{i=1}^n \bigcup_{k_i \geq 0} \ker(x_i \partial_{x_i}-\alpha_i)^{k_i}$ of $M$. There exists a finite set $A \subset [-1,0)^n$ such that $M_\alpha =0$ for $\alpha \not\in A+\mathbb{Z}^n$. If we set $M^\textnormal{alg}:=\oplus_\alpha M_\alpha$, the natural morphism $M^\textnormal{alg} \otimes_{\C[x_1,...,x_n]\langle \partial_{x_1},...,\partial_{x_n}\rangle} \mathscr{D}_X \rightarrow M$ is an isomorphism.

\bigskip

To be precise, only the case $n=2$ will occur in this paper. In the situations that we will consider, it will be possible to make the above decomposition explicit and then apply the general theory of Hodge modules of M. Saito. For a complete review of the 6 operations formalism for $\mathscr{D}$-modules, see \cite{Meb89}.

\medskip

\section{Proof of Theorem \ref{theo1}}

\label{section2}

\medskip

\noindent
\textbf{Steps of the proof.} Let us begin to list the different steps of the proof:

\smallskip

\begin{enumerate}

\item We write the middle convolution $\textnormal{MC}_{\lambda_0}(M)$ as an intermediate direct image by the sum map. By changing coordinates and projectivizing, we can consider the case of a proper projection.

\item We use a property of commutation between nearby cycles and projective direct image in the theory of Hodge modules, in order to carry out the local study of a nearby cycle sheaf.

\item To be in a normal crossing situation and use the results of the theory of Hodge modules, we perform a blow-up and make completely explicit the nearby cycle sheaf previously introduced (Lemma \ref{explicite}).

\item We take the monodromy and the Hodge filtrations into account, using the degeneration at $E_1$ of the Hodge to de Rham spectral sequence and the Riemann-Roch theorem (following \cite{DS13}) to get the expected theorem.

\end{enumerate}

\medskip

\noindent
\textbf{Geometric situation.} Let $s:\A^1_x \times \A^1_y \rightarrow \A^1_t$ be the sum map. We can change the coordinates so that $s$ becomes the projection onto the second factor and projectivize to get $\widetilde{s}:\mathbb{P}^1_x \times \mathbb{P}^1_t \rightarrow \mathbb{P}^1_t$. We set $x'=1/x$ and $t'=1/t$ as coordinates on a neighbourhood of $(\infty,\infty) \in \mathbb{P}^1_x \times \mathbb{P}^1_t$, $M_{\lambda_0}=M \boxtimes \mathcal{K}_{\lambda_0}$ and $\mathscr{M}_{\lambda_0}=(M_{\lambda_0})_{\textnormal{min}(x'=0)}$ the minimal extension of $M_{\lambda_0}$ along the divisor $\{x'=0\}$. A reasoning similar to that of \cite[Prop 1.1.10]{DS13} gives $\textnormal{MC}_{\lambda_0}(M)=\widetilde{s}_+ \mathscr{M}_{\lambda_0}$. With some abuse of notation, we still denote by $M$ the push-forward in the sense of $\mathscr{D}$-modules by the inclusion $\A^1_x \hookrightarrow \mathbb{P}^1_x$ and we denote by $M'$ its restriction to the affine chart centered at $\infty$. A similar abuse of notation is made for $M_{\lambda_0}$.

\medskip

Let us specify the geometric situation that we will consider in the following, in which we blow up the point $(\infty,\infty)$ in $\mathbb{P}^1_x \times \mathbb{P}^1_t$. We set $X=\textnormal{Bl}_{(\infty,\infty)}(\mathbb{P}^1_x \times \mathbb{P}^1_t)$, $e : X \rightarrow \mathbb{P}^1_x \times \mathbb{P}^1_t$ and $j:\A^1_x \times \A^1_t \hookrightarrow X$ the natural inclusion. There are two charts of the blow up : one given by coordinates $(u_1,v_1) \mapsto (t'=u_1 v_1,x'=v_1)$ and the other one by $(u_2,v_2) \mapsto (t'=v_2,x'=u_2 v_2)$. The strict transform of the line $\{t'=0\}$ is called $\mathbb{P}_x^1$, and the exceptional divisor is called $\mathbb{P}_{\textnormal{exc}}^1$. We denote by $0 \in \mathbb{P}_{\textnormal{exc}}^1$ the point given by $u_2=0$, $1 \in \mathbb{P}_{\textnormal{exc}}^1$ the point given by $u_2=1$ and $\infty \in \mathbb{P}_x^1 \cap \mathbb{P}_{\textnormal{exc}}^1$. We have the following picture:

\vspace{0.3cm}

\includegraphics[scale=0.43]{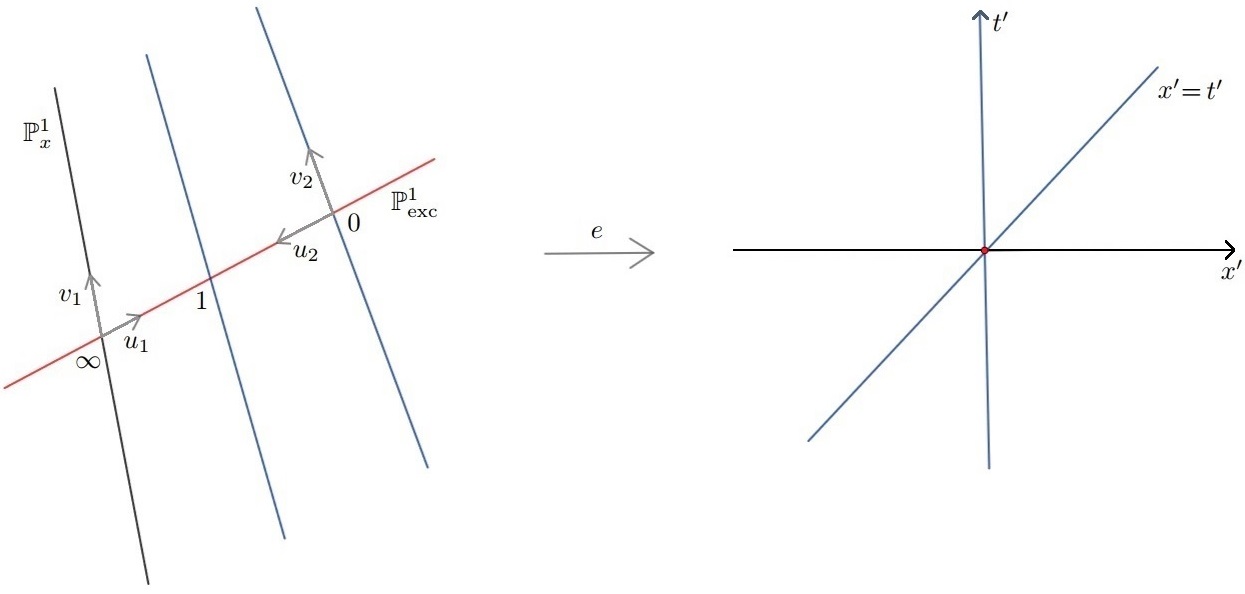}

\vspace{0.2cm}

\noindent
On $\mathbb{A}^1_x \times \mathbb{A}^1_t$, we have

\vspace{-0.5cm}

\begin{align*}
M_{\lambda_0}=M \boxtimes \mathcal{K}_{\lambda_0} &= M[t] \otimes \left( \C[x,t,(t-x)^{-1}],\textnormal{d}_{(x,t)}+\gamma_0 \frac{\textnormal{d}(t-x)}{t-x} \right) \\[-1mm]
&= \left( M[t,(t-x)^{-1}],\nabla_{(x,t)}+\gamma_0 \frac{\textnormal{d}(t-x)}{t-x} \right).
\end{align*}

\noindent
On the affine chart centered at $(\infty,\infty)$, we have

\vspace{-0.4cm}

$$M_{\lambda_0}= \left( M'[t',t'^{-1},(x'-t')^{-1}], \nabla_{(x',t')}+\gamma_0 \left(-\frac{\textnormal{d}x'}{x'}-\frac{\textnormal{d}t'}{t'}+\frac{\textnormal{d}(x'-t')}{x'-t'}\right) \! \right) \! .$$

\medskip

\begin{enonce}{Notation} \em Let us set $N_{\lambda_0}=e^+ M_{\lambda_0}$, $\mathscr{N}_{\lambda_0}=(N_{\lambda_0})_{\textnormal{min}(x' \circ e=0)}$ and $T^\lambda = \psi_{t' \circ e, \lambda}\mathscr{N}_{\lambda_0}$ equipped with a nilpotent endomorphim denoted by $\mathrm{N}$.\em
\end{enonce}

\begin{lemm} $\mathscr{M}_{\lambda_0}[t'^{-1}]=e_+\mathscr{N}_{\lambda_0}[t'^{-1}]$.
\end{lemm}

\begin{proof} By definition of the minimal extension, $\mathscr{N}_{\lambda_0}$ is the image of the map $j_\dagger j^+ N_{\lambda_0} \longrightarrow j_+ j^+ N_{\lambda_0}$. For $i \neq 0$, $H^i e_+ \mathscr{N}_{\lambda_0}$ is supported on $(\infty,\infty)$, thereby $H^i e_+ \mathscr{N}_{\lambda_0}[t'^{-1}]=0$. As the kernel of $H^0 e_+ \mathscr{N}_{\lambda_0} \longrightarrow M_{\lambda_0}$ is similarly supported on $(\infty,\infty)$, we deduce that $e_+ \mathscr{N}_{\lambda_0}[t'^{-1}]$ is a submodule of $M_{\lambda_0}$.

\vspace{0.1cm}

\noindent
We have $e_+ j_+ j^+ N_{\lambda_0}=(e \circ j)_+ (e \circ j)^+ M_{\lambda_0}$ and, as $e$ is proper, we can write $e_+ j_\dagger j^+ N_{\lambda_0}=e_\dagger j_\dagger j^+ N_{\lambda_0}=(e \circ j)_\dagger (e \circ j)^+ M_{\lambda_0}$. Then $\mathscr{M}_{\lambda_0}[t'^{-1}]$ is the image of the map $e_+ j_\dagger j^+ N_{\lambda_0}[t'^{-1}] \longrightarrow e_+ j_+ j^+ N_{\lambda_0}[t'^{-1}]$. Outside of $(\infty,\infty)$, $\mathscr{M}_{\lambda_0}[t'^{-1}]$ and $e_+\mathscr{N}_{\lambda_0}[t'^{-1}]$ are submodules of $M_{\lambda_0}$ which are isomorphic. Now, we can consider the intersection of these two submodules of $M_{\lambda_0}$, with two morphisms from the intersection to each of them. The kernel and the cokernel of these two morphisms are a priori supported on $(\infty,\infty)$, but as $t'$ is invertible, they are zero. Therefore, $\mathscr{M}_{\lambda_0}[t'^{-1}]$ and $e_+\mathscr{N}_{\lambda_0}[t'^{-1}]$ are isomorphic.
\end{proof}

\vspace{0.1cm}

Let us fix $\gamma \in [0,1)$ and $\lambda=\exp(-2i\pi\gamma)$. As $\widetilde{s}$ and $e$ are proper, the nearby cycles functor is compatible with $(\widetilde{s} \circ e)_+$ \cite[Prop 3.3.17]{Sai88}, so we get

\vspace{-0.35cm}

$$\psi_{\infty,\lambda}(\textnormal{MC}_{\lambda_0}(M))=\psi_{t',\lambda}(\widetilde{s}_+ \mathscr{M}_{\lambda_0})=\psi_{t',\lambda}(\widetilde{s}_+ e_+ \mathscr{N}_{\lambda_0})=\widetilde{s}_+ e_+ T^\lambda.$$

\smallskip

\begin{lemm}\label{explicite}
We set $(M')_\gamma = \ker(x'\partial_{x'}-\gamma)^r$ acting on $\psi_{x'}M'$ for $r \gg 0$ and $(M')^{\lambda}=(M')_\gamma[x',x'^{-1}]$. Let us set
$$T_0^\lambda=\left( (M')^{\lambda \lambda_0}[(x'-1)^{-1}], \nabla+\gamma_0 \left(-\frac{\textnormal{d}x'}{x'}+\frac{\textnormal{d}x'}{x'-1} \right) \right)$$
in the chart $\mathbb{P}_{\textnormal{exc}}^1 \setminus \{\infty\}$ (with the coordinate $x'$ instead of $u_2$) and denote simi-\\larly its meromorphic extension at infinity, with the action of the nilpotent endomorphism $x' \partial_{x'}-(\gamma+\gamma_0)$. For the extension by zero (instead of meromorphic), we use the notation $(T_0^\lambda)'$. Then:

\vspace{0.15cm}

\noindent
(Case 1) For $\lambda \not \in \{1,\bar{\lambda_0}\}$, $T^\lambda$ is supported on $\mathbb{P}_{\textnormal{exc}}^1$ and $T^\lambda=T_0^\lambda$.

\vspace{0.15cm}

\noindent
(Case 2) For $\lambda=1$, $T^\lambda$ is supported on $\mathbb{P}_{\textnormal{exc}}^1$ and is isomorphic to the minimal extension of $T_0^1$ at $x'=0$.

\vspace{0.15cm}

\noindent
(Case 3) For $\lambda=\bar{\lambda_0}$, $T^\lambda$ is supported on $\mathbb{P}_x^1 \cup \mathbb{P}_{\textnormal{exc}}^1$ and comes in an exact sequence
$$0 \rightarrow (T_0^{\bar{\lambda_0}})' \rightarrow T^\lambda \rightarrow T_1 \rightarrow 0$$
compatible with the nilpotent endomorphism, where $T_1$ is supported on $\mathbb{P}_x^1$ and is isomorphic to the meromorphic extension of $M$ at infinity with the action by 0 of the nilpotent endomorphism.
\end{lemm}

\begin{proof} We make a local study of the problem, reasoning in the three follo\-wing charts:\\
(i) in the chart $(u_2,v_2)$, called \textit{Chart 1};\\
(ii) in the neighbourhood of $\mathbb{P}_x^1 \setminus \{\infty\}$, called \textit{Chart 2};\\
(iii) in the neighbourhood of $\infty$, called \textit{Chart 3}.

\vspace{0.15cm}

\noindent
The cases of Charts 1 and 2 do not contain any significant problem and are treated in \cite[4.2.4]{Mar18}. In Chart 1, we find
$$\psi_{v_2,\lambda}\mathscr{N}_{\lambda_0}=\left\{
\begin{array}{cl}
  T_0^\lambda \quad &\textnormal{if } \lambda \neq 1 \\[1mm]
	(T_0^\lambda)_{\textnormal{min}(\{0\})} \quad &\textnormal{if } \lambda=1.
\end{array}
\right.$$

\noindent
In Chart 2, in which one can use the coordinates $(x,t')$, we find
$$\psi_{t',\lambda}\mathscr{M}_{\lambda_0}=\left\{
\begin{array}{cl}
  0 \quad &\textnormal{if } \lambda \neq \bar{\lambda_0} \\[1mm]
	M \quad &\textnormal{if } \lambda=\bar{\lambda_0}.
\end{array}
\right.$$

\noindent
Let us make the case of Chart 3 precise. For $\bm{\alpha}=(\alpha_1,\alpha_2)\in \R^2$, we set
$$(N_{\lambda_0})_{\bm{\alpha}}=\bigcup_{r_1 \geq 0}\ker(u_1\partial_{u_1} -\alpha_1)^{r_1} \cap \bigcup_{r_2 \geq 0}\ker(v_1\partial_{v_1} -\alpha_2)^{r_2}.$$

\noindent
By writing the expression of the connection in coordinates $(u_1,v_1)$, we get that $(N_{\lambda_0})_{(-\gamma_0,\alpha-\gamma_0)}$ can be identified with $(M')_\alpha$ with actions of $u_1 \partial_{u_1}$ and $v_1 \partial_{v_1}$ respectively expressed as $-\gamma_0 \textnormal{Id}$ and $x'\partial_{x'}-\gamma_0 \textnormal{Id}$. Here $\psi_{t' \circ e,\lambda}\mathscr{N}_{\lambda_0}=\psi_{g,\lambda}\mathscr{N}_{\lambda_0}$ where $g=u_1 v_1$, and we are in the situation of a calculation of nearby cycles of a coherent $\C[u_1,v_1]\langle \partial_{u_1},\partial_{v_1} \rangle$-module of normal crossing type along $D=\{g=0\}$ where $g=u_1 v_1$ is a monomial function. The general question is developed in \cite[{\S}3.a]{Sai90} (see also \cite[{\S}13.3]{SS19}); let us make it precise in the particular case at hand. We consider the commutative diagram

\vspace{-0.2cm}

$$\xymatrix  @!0 @R=4em @C=4.5pc {
    X \ \ar@{^{(}->}[r]^-{i_g} \ar[rd]_g & X \times \mathbb{C}_z \ar@{->>}[d]^{p_2} \\
     & \mathbb{C}_z
  }
$$

\vspace{0.1cm}

\noindent
where $i_g$ is the graph embedding.

\medskip

\noindent
Then we can see that $(i_g)_+ \mathscr{N}_{\lambda_0}=\mathscr{N}_{\lambda_0}[\partial_z]=\bigoplus_{k \geq 0} \: (\mathscr{N}_{\lambda_0} \otimes \partial_z^k)$ is a left $\mathbb{C}[u_1,v_1,z] \langle \partial_{u_1},\partial_{v_1},\partial_z \rangle$-module equipped with the following actions:

\noindent
(i) action of $\mathbb{C}[u_1,v_1]$ : $f(u_1,v_1) \cdot (m \otimes \partial_z^k)=(f(u_1,v_1) m)\otimes \partial_z^k$.\\
(ii) action of $\partial_z$ : $\partial_z (m \otimes \partial_z^k)=m \otimes \partial_z^{k+1}$.\\
(iii) action of $\partial_{u_1}$ : $\partial_{u_1} (m \otimes \partial_z^k)=(\partial_{u_1}m) \otimes \partial_z^k-v_1 m \otimes \partial_z^{k+1}$.\\
(iv) action of $\partial_{v_1}$ : $\partial_{v_1} (m \otimes \partial_z^k)=(\partial_{v_1}m) \otimes \partial_z^k-u_1 m \otimes \partial_z^{k+1}$.\\
(v) action of $z$ : $z \cdot (m \otimes \partial_z^k)= u_1 v_1 m \otimes \partial_z^k - km \otimes \partial_z^{k-1}$.

\medskip

\noindent
Let us denote by $S_u$ (resp. $S_v$) the operator defined by $S_u(m \otimes \partial_z^k) \! = \! (u_1 \partial_{u_1}m) \otimes \partial_{z}^k$ (resp. $S_v(m \otimes \partial_z^k) = (v_1 \partial_{v_1}m) \otimes \partial_{z}^k$). With $E=z \partial_z$, we get the relations

\vspace{-0.1cm}

$$u_1 \partial_{u_1}(m \otimes \partial_z^k)=(S_u-E-(k+1))(m \otimes \partial_{z}^k)$$

\vspace{-0.1cm}

\noindent
and

\vspace{-0.2cm}

$$v_1 \partial_{v_1}(m \otimes \partial_z^k)=(S_v-E-(k+1))(m \otimes \partial_{z}^k).$$

\bigskip
\medskip

\noindent
Letting $V^{\bullet}(\mathscr{N}_{\lambda_0}[\partial_z])$ denote the $V$-filtration with respect to $z$, we have $T=\psi_{g,\lambda} \mathscr{N}_{\lambda_0}=\textnormal{gr}_V^{\gamma}(\mathscr{N}_{\lambda_0}[\partial_z])$. We have the decompositions $\mathscr{N}_{\lambda_0}^{\textnormal{alg}}=\bigoplus_{\bm{\alpha} \in \mathbb{R}^2} (\mathscr{N}_{\lambda_0})_{\bm{\alpha}}$ and $(T^\lambda)^{\textnormal{alg}}=\bigoplus_{\bm{\beta} \in \mathbb{R}^2} T_{\bm{\beta}}$, and by arguing in a way similar to \cite[Lemma 13.2.26]{SS19} (with left modules), we show that the only indices $\bm{\beta}$ that appear are those such that $\bm{\alpha}=\bm{\beta}+(\gamma+k+1)(1,1)$ for $\bm{\alpha}$ in the decomposition of $\mathscr{N}_{\lambda_0}^{\textnormal{alg}}$ and $k \in \mathbb{Z}$. In particular, we cannot have $\alpha_2 \in \mathbb{Z}$ and having a minimal extension along $\{ v_1=0 \}$ does not play any role here. In other words, we can identify in this part $\mathscr{N}_{\lambda_0}$ and $N_{\lambda_0}$.

\bigskip

\noindent
More precisely, for $\beta_1,\beta_2 \geq -1$, we deduce from \cite[Th.\,3.3]{Sai90} (or \cite[Cor. 13.2.32]{SS19}) the following expressions for $T_{\bm{\beta}}$:

\vspace{-0.05cm}

$$T_{\bm{\beta}}=\left\{
\begin{array}{cl}
  0 \quad &\textnormal{if } \beta_1 \neq -1, \ \beta_2 \neq -1 \\[1mm]
  \coker(S_u-E \in \textnormal{End}((N_{\lambda_0})_{\gamma,\beta_2+\gamma+1}[E])) \quad &\textnormal{if } \beta_1 = -1, \ \beta_2 \neq -1 \\[1mm]
  \coker(S_v-E \in \textnormal{End}((N_{\lambda_0})_{\beta_1+\gamma+1,\gamma}[E])) \quad &\textnormal{if } \beta_1 \neq -1, \ \beta_2 = -1 \\[1mm]
  \coker((S_u-E)(S_v-E) \in \textnormal{End}((N_{\lambda_0})_{\gamma,\gamma}[E])) \quad &\textnormal{if } \bm{\beta}=(-1,-1).
\end{array}
\right.$$

\bigskip
\medskip

\noindent
In what follows, we will work with the point of view of quivers. A good reference can be found in Chapters 3.1.b and 9.4 of \cite{SS19}. Let $A \subset (-1,0]$ be the (finite) set of those $\alpha \in (-1,0]$ such that $(M')_\alpha \neq 0$ and let us look at the different cases for $\bm{\beta} \in [-1,0]^2$ :

\medskip

\noindent
(i) For $\beta_2 \neq -1$, we have $T_{(-1,\beta_2)} \neq 0$ if and only if $\gamma=-\gamma_0$ and $\beta_2 \in A$.\\
(ii) For $\beta_1 \neq -1$, we have $T_{(\beta_1,-1)} \neq 0$ if and only if $\gamma=\alpha-\gamma_0$ with $\alpha \in A \textnormal{ mod } \mathbb{Z}$ and $\beta_1=-\alpha \textnormal{ mod } \mathbb{Z}$.\\
(iii) $T_{(-1,-1)} \neq 0$ if and only if $\gamma=-\gamma_0$ and $0 \in A$.

\bigskip
\medskip

\noindent
We deduce from these relations that:

\bigskip

\noindent
(Case 1+2) If $\gamma \neq -\gamma_0$ then $T^\lambda$ is supported on $\mathbb{P}_{\textnormal{exc}}^1$ and, according to (ii), is determined by the only data of $\coker(S_v-E \in \textnormal{End}((N_{\lambda_0})_{-\gamma_0,\gamma}[E]))$ equipped with an action of $E-\gamma$, that we can identify with $(N_{\lambda_0})_{-\gamma_0,\gamma}$ where the action of $E-\gamma$ can be identified with $S_v-\gamma$. Consequently, $T^\lambda$ is determined by $(M')_{\gamma+\gamma_0}$ with an action of $x' \partial_{x'}-(\gamma+\gamma_0)$.

\bigskip

\noindent
(Case 3) If $\gamma=-\gamma_0$ then $T^\lambda$ is determined by two data:

\medskip

\noindent
$\bullet$ firstly by the data, according to (i), of $\coker(S_u-E \in \textnormal{End}((N_{\lambda_0})_{-\gamma_0,\alpha-\gamma_0}[E]))$ for $\alpha \in A \textnormal{ mod } \mathbb{Z}$, $\alpha \not\in \mathbb{Z}$, supported on $\mathbb{P}_{x}^1$ and equipped with an action of $E+\gamma_0$. We can identify it with $(N_{\lambda_0})_{-\gamma_0,\alpha-\gamma_0}$ where the action of $E+\gamma_0$ is identified with $S_u+\gamma_0$, that we identify with $(M')_{\alpha}$ with an action by $0$.

\medskip

\noindent
$\bullet$ secondly by the data of the biquiver

\vspace{-0.2cm}

$$\xymatrix{
    T_{(-1,-1)} \ \ar@<2pt>[r]^-{v_1} \ar@<2pt>[d]^-{u_1} & \ T_{(-1,0)} \ar@<2pt>[l]^-{\partial_{v_1}} \\
    T_{(0,-1)} \ar@<2pt>[u]^-{\partial_{u_1}} &
  }
$$

\medskip

\noindent
As $u_1^{-1}$ acts on $(N_{\lambda_0})_{-\gamma_0+1,-\gamma_0}[E]$ and $v_1^{-1}$ on $(N_{\lambda_0})_{-\gamma_0,-\gamma_0+1}[E]$, it can be assumed that $T_{(-1,-1)}$, $T_{(0,-1)}$ and $T_{(-1,0)}$ are all three cokernels of maps of $\textnormal{End}((N_{\lambda_0})_{-\gamma_0,-\gamma_0}[E])$. Now we set $C_{uv}=\coker(S_u-E)(S_v-E)$, $C_u=\coker(S_u-E)$ and $C_v=\coker(S_v-E)$, and the previous biquiver can be identified with the following

\vspace{-0.2cm}

$$\xymatrix{
    C_{uv} \ \ar@<2pt>@{->>}[r]^-{\varphi_v} \ar@<2pt>@{->>}[d]^-{\varphi_u} & \ C_u \ar@<2pt>[l]^-{S_v-E} \\
    C_v \ar@<2pt>[u]^-{S_u-E} &
  }
$$

\medskip

\noindent
where $\varphi_u : C_{uv} \rightarrow C_v$ is induced by the inclusion $\im(S_u-E)(S_v-E) \subset \im(S_v-E)$, and the same for $\varphi_v$. As $S_u-E \in \textnormal{End}((N_{\lambda_0})_{-\gamma_0,-\gamma_0}[E])$ is injective (because $S_u$ is identified on $(M')_0$ with $-\gamma_0 \textnormal{Id}$) and

$$\im(S_u-E:C_v \rightarrow C_{uv})=\frac{\im(S_u-E)}{\im(S_u-E)(S_v-E)}=\ker \: \varphi_v,$$

\bigskip

\noindent
we deduce the following exact sequence:

\vspace{0.05cm}

$$0 \rightarrow C_v \rightarrow C_{uv} \rightarrow C_u \rightarrow 0.$$

\bigskip

\noindent
Therefore, we have an exact sequence of biquivers:

\vspace{-0.2cm}

$$\hspace{-1cm}\xymatrix{
   0 \hspace{-1.8cm} & \longrightarrow \hspace{-1cm} & C_v \ \ar@<2pt>[r] \ar@<1pt>[d]^-{S_u-E} & \ 0  \ \ \longrightarrow \hspace{-0.5cm} \ar@<2pt>[l] & C_{uv} \ \ar@<2pt>@{->>}[r]^-{\varphi_v} \ar@<2pt>@{->>}[d]^-{\varphi_u} & \ C_u \ \ \longrightarrow \hspace{-0.6cm} \ar@<2pt>[l]^-{S_v-E} & \!\!\! C_u \ \ar@<2pt>[r]^-{\textnormal{Id}} \ar@<-3pt>[d] & \ C_u \ \longrightarrow \ 0 \ar@<2pt>[l]^-{S_v-E} \\
   & & \!\! C_v \ar@<3pt>[u]^-{\textnormal{Id}} & & C_v \ar@<2pt>[u]^-{S_u-E} & & 0 \quad \ar@<7pt>[u] &
  }
$$

\medskip

\noindent
$\bullet$ The left biquiver is a quiver representing the extension by zero supported on $\mathbb{P}_{\textnormal{exc}}^1$ and identified with $(N_{\lambda_0})_{-\gamma_0,-\gamma_0}$ where the action of $E+\gamma_0$ can be identified with $S_v+\gamma_0$, in other words $(M')_0$ with the action of $x' \partial_{x'}$. This is the following biquiver:

\vspace{-0.2cm}

$$\xymatrix{
    (M')_0 \ \ \ \ar@<2pt>[r] \ar@<-6pt>[d]^-{-x' \! \partial_{x'}} & \ \ \ 0 \ar@<2pt>[l] \\
    (M')_0 \ \ \ \ar@<10pt>[u]^-{\textnormal{Id}} & 
  }
$$

\vspace{0.1cm}

\noindent
$\bullet$ The right biquiver is a quiver representing the meromorphic extension supported on $\mathbb{P}_x^1$ and identified with $(N_{\lambda_0})_{-\gamma_0,-\gamma_0}$ where the action of $E+\gamma_0$ is equal to $0$, in other words $(M')_0$ with the action by $0$. This is the following biquiver:

\vspace{-0.5cm}

$$\xymatrix{
    (M')_0 \ \ \ar@<2pt>[r]^-{\textnormal{Id}} \ar@<-4pt>[d] & \ \ (M')_0 \ar@<2pt>[l]^-{x' \! \partial_{x'}} \\
    0 \ \ar[u] & 
  }
$$

\vspace{0.1cm}

\noindent
$\bullet$ It is possible to make the central biquiver explicit in terms of $(M')_0$, insofar as we can identify $C_{uv}$ with $((M')_0)^2$ with an action of

$$E+\gamma_0=\left (\begin{array}{c|c}
\gamma_0 \textnormal{Id} \ & \gamma_0(x' \! \partial_{x'}-\gamma_0 \textnormal{Id})  \\
\hline
\textnormal{Id} \ & x' \! \partial_{x'}-\gamma_0 \textnormal{Id}
\end{array}\right),$$

\medskip

\noindent
and we get the following biquiver:

\vspace{-0.3cm}

$$\xymatrix @!0 @R=1.8cm @C=3cm {
    ((M')_0)^2 \ \ \ \ \ar@<2pt>@{->>}[r]^-{p_v} \ar@<-6pt>@{->>}[d]^-{p_u} & \ \ \ \ (M')_0 \ar@<2pt>[l]^-{(x' \! \partial_{x'}-\gamma_0 \textnormal{Id},-\textnormal{Id})} \\
    \!\!\!\!\!\! (M')_0 \ar@<10pt>[u]^-{(-\gamma_0 \textnormal{Id},-\textnormal{Id})} & 
  }
$$

\noindent
where $p_u=p_1+(x' \! \partial_{x'}-\gamma_0)p_2$ and $p_v=p_1-\gamma_0 p_2$, with $p_1$ the projection onto the first factor and $p_2$ onto the second factor.

\noindent
Finally, for $\lambda=\bar{\lambda_0}$ we have an exact sequence
$$0 \rightarrow (T_0^{\bar{\lambda_0}})' \rightarrow T^\lambda \rightarrow T_1 \rightarrow 0$$
where $(T_0^{\bar{\lambda_0}})'$ is the extension by zero of $T_0^{\bar{\lambda_0}}$ at infinity equipped with the nilpotent endomorphism $x' \partial_{x'}$, and $T_1$ is supported on $\mathbb{P}_x^1$ and given by the meromorphic extension of $M$ at infinity equipped with the nilpotent endomorphism $0$.

\medskip

\noindent
By gluing the expressions obtained for the different values of $\lambda$ in each of the three charts, we get the announced result of the lemma.
\end{proof}

\medskip

By construction, the complex $K^\bullet=\widetilde{s}_+ e_+ T^\lambda$ has cohomology in degree zero only. More precisely, as $\widetilde{s} \circ e : \mathbb{P}_{\textnormal{exc}}^1 \cup \mathbb{P}_{x}^1 \rightarrow \{\textnormal{pt}\}$, that amounts to saying that $\textbf{R} \Gamma(\mathbb{P}_{\textnormal{exc}}^1 \cup \mathbb{P}_{x}^1,\DR^{\textnormal{an}}T^\lambda)$ is a two-term complex with a kernel reduced to zero. If we take the monodromy filtration $\mathrm{M}_\bullet$ into account, we have the following more precise result:

\vspace{0.2cm}

\begin{lemm}\label{hj}
$H^j(\textnormal{gr}_\ell^{\mathrm{M}} K^\bullet)=0$ for $j \neq 0$ and $\ell \in \mathbb{Z}$.
\end{lemm}

\begin{proof} (Case 1+2) If $\lambda \neq \bar{\lambda_0}$, then $T^\lambda$ is supported on $\mathbb{P}_{\textnormal{exc}}^1$ and localized at infinity. Consequently, we can see $T^\lambda$ as a $\C[x']\langle \partial_{x'}\rangle$-module, that is a $\C[x']$-module with a connection $\nabla_{\partial_{x'}}$. The question is to show that the morphism
$$\textnormal{gr}_{\ell}^{\mathrm{M}} T^\lambda \overset{\nabla_{\partial_{x'}}}{\longrightarrow} \textnormal{gr}_{\ell}^{\mathrm{M}} T^\lambda$$

\noindent
has a kernel reduced to zero. With $\mathrm{N}$ the nilpotent endomorphism, let us set
$$\bar{T}^\lambda=\left( (M')^{\lambda\lambda_0}, \bar{\nabla}_{\partial_{x'}}=\frac{\partial}{\partial x'}+\frac{\gamma}{x'}\textnormal{Id} +\frac{\mathrm{N}}{x'} \right)$$
\noindent
where $\bar{T}^\lambda$ is minimally extended at 0 if $\lambda=1$, and
$$1_{\gamma_0}=\left( \C[x',(x'-1)^{-1}], \frac{\partial}{\partial x'}+ \frac{\gamma_0}{x'-1} \textnormal{Id} \right)$$
so that $T^\lambda=\bar{T}^\lambda \otimes 1_{\gamma_0}$. For $m \in \bar{T}^\lambda$ and $m' \in \C[x',(x'-1)^{-1}]$, we have
$$\nabla_{\partial_{x'}}(m \otimes m')=\bar{\nabla}_{\partial_{x'}}(m) \otimes m' + m \otimes \left( \frac{\partial}{\partial x'}+ \frac{\gamma_0}{x'-1} \textnormal{Id} \right) m'.$$

\noindent
Now, we have $\textnormal{gr}_{\ell}^{\mathrm{M}} T^\lambda=\textnormal{gr}_{\ell}^{\mathrm{M}} \bar{T}^\lambda \otimes 1_{\gamma_0}$ and we want to show that
$$\ker \left( \textnormal{gr}_{\ell}^{\mathrm{M}} \bar{T}^\lambda \otimes 1_{\gamma_0} \overset{\nabla_{\partial_{x'}}}{\longrightarrow} \textnormal{gr}_{\ell}^{\mathrm{M}} \bar{T}^\lambda \otimes 1_{\gamma_0} \right)=0.$$

\noindent
For $m \in \textnormal{gr}_{\ell}^{\mathrm{M}} \bar{T}^\lambda$, let us notice the equality
$$\nabla_{\partial_{x'}} \left( m \otimes \frac{1}{(x'-1)^k} \right) =\bar{\nabla}_{\partial_{x'}}(m) \otimes \frac{1}{(x'-1)^k} - \frac{km}{(x'-1)^{k+1}} + \frac{\gamma_0 m}{(x'-1)^{k+1}},$$

\noindent
from which we deduce, as $\gamma_0 \not\in \mathbb{Z}$, that $m \otimes (x'-1)^{-k}$ has a pole at 1 of order $k+1$ if $m \neq 0$. Therefore, if an element $m \otimes m'$ is such that $\nabla_{\partial_{x'}}(m \otimes m')=0$, then $m=0$.

\vspace{0.4cm}

\noindent
(Case 3) If $\lambda = \bar{\lambda_0}$, we have the exact sequence $0 \rightarrow T_0 \rightarrow T^\lambda \rightarrow T_1 \rightarrow 0$ where $T_0=(T_0^{\bar{\lambda_0}})'$ is supported on $\mathbb{P}_{\textnormal{exc}}^1$ and $T_1$ is supported on $\mathbb{P}_x^1$. Equivalently, we can think in terms of primitive parts instead of graded parts, as we shall do.

\vspace{0.2cm}

\noindent
The reasoning of the previous case applies in the same way to $K_0^\bullet=\widetilde{s}_+ e_+ T_0$, yielding that the complexes $\prim_\ell K_0^\bullet$ are concentrated in degree 0 for all $\ell \in \mathbb{N}$. As $\mathrm{N}$ is zero on $T_1$, we have $\mathrm{N}(T^\lambda) \subset T_0$. Let us show we have equality and, for that, let us go back to the description of $T^\lambda$ at the neighbourhood of $\infty$ in terms of biquivers:

\vspace{-0.4cm}

$$\xymatrix @!0 @R=1.5cm @C=3cm {
    ((M')_0)^2 \ \ \ \ \ar@<2pt>@{->>}[r]^-{p_v} \ar@<-6pt>@{->>}[d]^-{p_u} & \ \ \ \ (M')_0 \ar@<2pt>[l]^-{(x' \! \partial_{x'}-\gamma_0 \textnormal{Id},-\textnormal{Id})} \\
    \!\!\!\!\!\! (M')_0 \ar@<10pt>[u]^-{(-\gamma_0 \textnormal{Id},-\textnormal{Id})} & 
  }
$$

\medskip

\noindent
The action of $E+\gamma_0$ on $((M')_0)^2$ is given by
$$E+\gamma_0 = \left (\begin{array}{c|c}
\gamma_0 \textnormal{Id} \ & \gamma_0(x' \partial_{x'}-\gamma_0 \textnormal{Id})  \\
\hline
\textnormal{Id} \ & x' \partial_{x'}-\gamma_0 \textnormal{Id}
\end{array}\right),$$

\noindent
whose rank is equal to the dimension of $(M')_0$, and thus
$$(E+\gamma_0)((M')_0)^2 = \{(\gamma_0 m,m) \in ((M')_0)^2 \ | \ m \in (M')_0 \} \simeq (M')_0.$$

\medskip

\noindent
Now, a calculation shows that the following diagram

\vspace{-0.2cm}

$$\xymatrix @!0 @R=1.5cm @C=3cm {
    \!\!\!\!\!\!\!\!\!\!\!\! ((M')_0)^2 \quad \ar@{->>}[r]^-{p_2 \circ (E+\gamma_0)} \ar@<-6pt>@{->>}[d]^-{p_u} & \quad (M')_0 \ar@<2pt>@{->>}[d]^-{-x' \! \partial_{x'}} \\
    \!\!\!\!\!\! (M')_0 \quad \ar@<10pt>[u]^-{(-\gamma_0 \textnormal{Id},-\textnormal{Id})} \ar@{->>}[r]^-{-x' \! \partial_{x'}} & \quad (M')_0 \ar@<2pt>[u]^-{\textnormal{Id}}
  }
$$

\noindent
is commutative, so the image by $\mathrm{N}$ of the biquiver representing $T^\lambda$ gives the biquiver representing $T_0$, in other words $\mathrm{N}(T^\lambda)=T_0$. Consequently, we are in a situation of a minimal extension quiver:

\vspace{-0.35cm}

$$\xymatrix{
T^\lambda \quad \ar@<4pt>@{->>}[r]^-{\mathrm{N}} & \quad T_0. \ar@<4pt>@{_{(}->}[l]
}
$$

\medskip

\noindent
We deduce from \cite[Prop. 2.1.1(iii)]{KK87} that $\prim_\ell T^\lambda \simeq \prim_{\ell-1} T_0$ for $\ell \geq 1$, and thus the complexes $\prim_\ell K^\bullet$ are concentrated in degree 0 for all $\ell \geq 1$. Moreover, as the total complex $K^\bullet$ is concentrated in degree 0, the complex $P_0 K^\bullet$ is also concentrated in degree 0. We have the same property for graded parts instead of primitive parts.
\end{proof}

\vspace{0.1cm}

Now, with these two lemmas, we are able to show the main theorem. We can equip $T^\lambda$ with a Hodge filtration by properties of nearby cycles, and $(M')^{\lambda\lambda_0}$ as well, in such a way that the formula of Lemma \ref{explicite} is compatible with the Hodge filtrations. Note that these two objects are not variations of polarized Hodge structures, but variations of mixed Hodge structures. However, one can adapt the arguments of \cite{DS13} to the generalized context below.

\vspace{0.2cm}

\begin{rema}\label{remarque}
Let us begin making the Hodge filtration more explicit. In the chart $(u_2,v_2)$, let us consider the $\mathcal{D}$-module $e^+(M'\boxtimes\mathbb{C}[t',t^{\prime-1}])$, on which $u_2$ and $v_2$ act in an invertible way. This is
\[
e^*(M'\boxtimes\mathbb{C}[t',t^{\prime-1}])=\mathbb{C}[u_2,u_2^{-1},v_2,v_2^{-1}]\otimes_{\mathbb{C}[x',x^{\prime-1}]}M',
\]
on which the action of the connection is induced by
\[
u_2\partial_{u_2}(1\otimes m)=v_2\partial_{v_2}(1\otimes m)=x'\partial_{x'}(1\otimes m).
\]
\noindent
We consider the localized Hodge filtration $\widetilde F^pM'=\mathbb{C}[x',x^{\prime-1}]\otimes_{\mathbb{C}[x']}F^pM'$, $\widetilde F^p(M'\boxtimes\mathbb{C}[t',t^{\prime-1}])=(\widetilde F^pM')\boxtimes\mathbb{C}[t',t^{\prime-1}]$ and its pull-back $\widetilde F^pe^*(M'\boxtimes\mathbb{C}[t',t^{\prime-1}])=\mathbb{C}[u_2,u_2^{-1},v_2,v_2^{-1}]\otimes_{\mathbb{C}[x',x^{\prime-1}]}\widetilde F^pM'$. In order to compute nearby cycles along $v_2$, we make use of the $V$-filtration $V^\bullet(e^+(M'\boxtimes\mathbb{C}[t',t^{\prime-1}]))$ along $v_2$, that we will compare with the $V$-filtration of $M'$ along $x'$ denoted by $V^\bullet(M')$. The above relations link the Bernstein polynomial with respect to $v_2$ with that with respect to $x'$ and lead to the identification
\[
\mathbb{C}[u_2,u_2^{-1}]\otimes_{\mathbb{C}[u_2]}V^a(e^+(M'\boxtimes\mathbb{C}[t',t^{\prime-1}]))=\mathbb{C}[u_2,u_2^{-1},v_2]\otimes_{\mathbb{C}[x']}V^a M'.
\]
For $\lambda=e^{-2i\pi a}$, let us define the functor $\widetilde\psi_{v_2,\lambda}$ as the functor $\psi_{v_2,\lambda}$ followed by localization $\mathbb{C}[u_2,u_2^{-1}]\otimes_{\mathbb{C}[u_2]}\bullet \: $. We thus find an isomorphism of filtered $\mathbb{C}[u_2,u_2^{-1}]\langle u_2\partial_{u_2}\rangle$-modules
\begin{multline*}
(\widetilde\psi_{v_2,\lambda}e^+(M'\boxtimes\mathbb{C}[t',t^{\prime-1}]),\widetilde F^\bullet \widetilde\psi_{v_2,\lambda}e^+(M'\boxtimes\mathbb{C}[t',t^{\prime-1}]))\\\simeq\mathbb{C}[u_2,u_2^{-1}]\otimes_\mathbb{C}(\psi_{x',\lambda}M', F^\bullet\psi_{x',\lambda}M'),
\end{multline*}
since, for $a\in(-1,0]$, the filtration induced by $\widetilde F^\bullet M'$ on $\mathrm{gr}^a_VM'$ is equal to that induced by $F^\bullet M'$, and
where the action of $u_2\partial_{u_2}$ on the right-hand side is induced by $u_2\partial_{u_2}(1\otimes m)=0$.

\vspace{0.2cm}

\noindent
Considering instead the pull-back of $M_{\lambda_0}$ amounts to twisting $e^+(M'\boxtimes\mathbb{C}[t',t^{\prime-1}])$ by the pull-back connection $e^+(\mathbb{C}[x',x^{\prime-1}]\boxtimes\mathcal K_{\lambda_0})$. This leads to also invert the action of $u_2-1$ and to localize the $\widetilde F$-filtration along the divisor $u_2=1$. On the other hand, it twists the monodromies around $u_2=0$ and $v_2=0$ by $\lambda_0$. The localization of the Hodge filtration $F^\bullet T^\lambda$ is by definition the localization of the filtration induced by $\widetilde F^\bullet N_{\lambda_0}$ on $\psi_{v_2,\lambda}N_{\lambda_0}$. As a consequence, if we set $\widetilde T^\lambda=\mathbb{C}[u_2,u_2^{-1},(u_2-1)^{-1}]\otimes_{\mathbb{C}[u_2]}T^\lambda$ that we endow with the localized filtration $\widetilde F^\bullet \widetilde T^\lambda=\mathbb{C}[u_2,u_2^{-1},(u_2-1)^{-1}]\otimes_{\mathbb{C}[u_2]}F^\bullet T^\lambda$, we obtain an isomorphism (notation of Lemma \ref{explicite})

\vspace{-0.3cm}

\begin{equation}\label{iso}
(\widetilde T^\lambda, \widetilde F^\bullet \widetilde T^\lambda)\simeq (T_0^\lambda, \widetilde F^\bullet T_0^\lambda),
\end{equation}

\noindent
where $\widetilde F^\bullet T_0^\lambda$ is defined by $\widetilde F^p T_0^\lambda=\mathbb{C}[u_2,u_2^{-1},(u_2-1)^{-1}]\otimes_{\mathbb{C}}(F^p\psi_{x',\lambda\lambda_0}M')$.
\end{rema}

\vspace{0.2cm}

\begin{proof}[Proof of Theorem \ref{theo1}] (Case 1) Let us begin with the case $\lambda \not\in \{1,\bar{\lambda_0}\}$ and, as a first step, without taking the monodromy filtration into account. As $\widetilde{s} \circ e : \mathbb{P}_{\textnormal{exc}}^1 \rightarrow \{\textnormal{pt}\}$, we have $\psi_{\infty,\lambda}(\textnormal{MC}_{\lambda_0}(M))=\widetilde{s}_+ e_+ T^\lambda=\textbf{R} \Gamma(\mathbb{P}_{\textnormal{exc}}^1,\DR T^\lambda)$. Setting $\bm{x}=\{0,1,\infty\}$ and $\DR T^\lambda=i_* \mathscr{V}$, where $i:\mathbb{P}_{\textnormal{exc}}^1 \setminus \bm{x} \hookrightarrow \mathbb{P}_{\textnormal{exc}}^1$ denotes the open inclusion and $\mathscr{V}$ is a non-constant local system on $\mathbb{P}_{\textnormal{exc}}^1 \setminus \bm{x}$, we have $\bm{H}^m(\mathbb{P}_{\textnormal{exc}}^1,\DR T^\lambda)=H^m(\mathbb{P}_{\textnormal{exc}}^1,i_* \mathscr{V})$. This quantity is zero for $m=0$ because there is no section supported on a point, and $\mathscr{V}$ has no global section since one of its local monodromies has not $1$ as an eigenvalue. Using the Poincar\'e duality theorem, we have also $H^2(\mathbb{P}_{\textnormal{exc}}^1,i_* \mathscr{V})=0$ if we argue with the dual $\mathscr{V}^\vee$. Then

\vspace{-0.35cm}

$$\nu_{\infty,\lambda}^p(\textnormal{MC}_{\lambda_0}(M))=\dim \textnormal{gr}_F^p \psi_{\infty,\lambda}(\textnormal{MC}_{\lambda_0}(M))=\dim \textnormal{gr}_F^p \bm{H}^1(\mathbb{P}_{\textnormal{exc}}^1,\DR T^\lambda).$$

\vspace{0.1cm}

\noindent
Adapting \cite[Prop. 2.3.3]{DS13} (we use the formula given at the end of their proof, without using $\nu_{\infty,1,\textnormal{prim}}^{p-1}(T^\lambda)$), we have

\vspace{-0.4cm}

\begin{align}\label{eq1}
\nu_{\infty,\lambda}^p(\textnormal{MC}_{\lambda_0}(M))&=\delta^{p-1}(T^\lambda)-\delta^p(T^\lambda) - h^p(T^\lambda)- h^{p-1}(T^\lambda) \\[-1mm]
& \hspace{2.5cm} + \sum\limits_{x \in \bm{x}} \Bigg( \underbrace{\sum\limits_{\mu \neq 1} \nu_{x,\mu}^{p-1}(T^\lambda)}_{=h^{p-1}(T^\lambda)} + \underbrace{\mu_{x,1}^p(T^\lambda)}_{=0} \Bigg).\notag \\
&= \delta^{p-1}(T^\lambda)-\delta^p(T^\lambda) - h^p(T^\lambda)+2 h^{p-1}(T^\lambda) \notag
\end{align}

\noindent
since we remark that for each of the three singular points of $T^\lambda$, we have a single eigenvalue for the local monodromy (different from $1$) and only one term $\nu_{x,\mu}^{p-1}(T^\lambda)$ is non zero, hence equal to $h^{p-1}(T^\lambda)$.

\vspace{0.2cm}

\noindent
Now, as $\lambda \lambda_0 \neq 1$, we have $\bm{H}^1(\mathbb{P}_x^1,\DR (M')^{\lambda \lambda_0}) = 0$. The same reasoning with $(M')^{\lambda \lambda_0}$ instead of $T^\lambda$, which has two singularities, gives

\vspace{-0.3cm}

\begin{equation}\label{eq2}
0=\delta^{p-1}((M')^{\lambda \lambda_0})-\delta^p((M')^{\lambda \lambda_0}) - h^p((M')^{\lambda \lambda_0}) + h^{p-1}((M')^{\lambda \lambda_0}).
\end{equation}

\vspace{0.2cm}

\noindent
In accordance with Lemma \ref{explicite}, we have

\vspace{-0.3cm}

$$T^\lambda=(M')^{\lambda \lambda_0} \otimes \left( \C[x',x'^{-1},(x'-1)^{-1}], \textnormal{d}+\gamma_0 \left(-\frac{\textnormal{d}x'}{x'}+\frac{\textnormal{d}x'}{x'-1} \right) \right),$$

\vspace{0.05cm}

\noindent
and then we can infer, according to \cite[2.2.12]{DS13}, that $h^p(T^\lambda)=h^p((M')^{\lambda \lambda_0})$. Let us apply the formula for the twist by a unitary rank-one connection given in \cite[Prop. 2.3.2]{DS13}:

\vspace{0.1cm}

\begin{multline}
\delta^p(T^\lambda)=\delta^p((M')^{\lambda \lambda_0})-h^p((M')^{\lambda \lambda_0})+\sum_{\alpha \in [\gamma_0,1)} \nu_{\infty,e^{-2i\pi\alpha}}^p((M')^{\lambda \lambda_0}) \\
+\sum_{\alpha \in [1-\gamma_0,1)} \underbrace{\nu_{1,e^{-2i\pi\alpha}}^p((M')^{\lambda \lambda_0})}_{=0 \ \textnormal{because } \alpha \neq 0}.
\end{multline}

\vspace{-0.15cm}

\noindent
The remaining sum can be expressed as

\vspace{-0.25cm}

$$\sum_{\alpha \in [\gamma_0,1)} \nu_{\infty,e^{-2i\pi\alpha}}^p((M')^{\lambda \lambda_0})=\left\{
\begin{array}{cl}
  h^p((M')^{\lambda \lambda_0}) \ &\textnormal{if } \gamma \in (0,1-\gamma_0) \\[1mm]
  0 \ &\textnormal{if } \gamma \in (1-\gamma_0,1),
\end{array}
\right.
$$

\vspace{0.15cm}

\noindent
so we deduce that

\vspace{-0.25cm}

\begin{equation}\label{eq3}
\delta^p(T^\lambda)=
\left\{
\begin{array}{cl}
  \delta^p((M')^{\lambda \lambda_0}) \ &\textnormal{if } \gamma \in (0,1-\gamma_0) \\[1mm]
  \delta^p((M')^{\lambda \lambda_0})-h^p((M')^{\lambda \lambda_0}) \ &\textnormal{if } \gamma \in (1-\gamma_0,1).
\end{array}
\right.
\end{equation}

\vspace{0.15cm}

\noindent
We are now able to apply the formula (\ref{eq1}):\\
(i) For $\gamma \in (0,1-\gamma_0)$, we have:

\vspace{-0.35cm}

\begin{align*}
\nu_{\infty,\lambda}^p(\textnormal{MC}_{\lambda_0}(M)) &= \delta^{p-1}(T^\lambda)-\delta^p(T^\lambda) - h^p(T^\lambda) + 2h^{p-1}(T^\lambda) \\
&= \delta^{p-1}((M')^{\lambda \lambda_0})-\delta^p((M')^{\lambda \lambda_0})-h^p((M')^{\lambda \lambda_0}) \\
& \hspace{4cm} +2h^{p-1}((M')^{\lambda \lambda_0}) \\
&= h^{p-1}((M')^{\lambda \lambda_0}) \ \ \textnormal{(according to } \ref{eq2} \textnormal{)}.
\end{align*}

\vspace{0.25cm}

\noindent
(ii) For $\gamma \in (1-\gamma_0,1)$, we have:

\vspace{-0.35cm}

\begin{align*}
\nu_{\infty,\lambda}^p(\textnormal{MC}_{\lambda_0}(M)) &= \notag \\
& \hspace{-1.8cm} \bigg( \delta^{p-1}((M')^{\lambda \lambda_0})-h^{p-1}((M')^{\lambda \lambda_0}) \bigg) - \bigg( \delta^p((M')^{\lambda \lambda_0})-h^{p}((M')^{\lambda \lambda_0}) \bigg) \notag \\
& \hspace{2.5cm} -h^{p}((M')^{\lambda \lambda_0})+2h^{p-1}((M')^{\lambda \lambda_0}) \\
&= \delta^{p-1}((M')^{\lambda \lambda_0})-\delta^p((M')^{\lambda \lambda_0})+h^{p-1}((M')^{\lambda \lambda_0}) \\
&= h^{p}((M')^{\lambda \lambda_0}) \ \ \textnormal{(according to } \ref{eq2} \textnormal{)}.
\end{align*}

\vspace{0.2cm}

\noindent
We deduce from the isomorphism given in \ref{iso} that $h^p(T^\lambda)=\nu_{\infty,\lambda\lambda_0}^{p}(M)$, then

\vspace{-0.2cm}

$$h^{p}((M')^{\lambda \lambda_0})=\nu_{\infty,\lambda\lambda_0}^{p}(M).$$

\vspace{0.2cm}

\noindent
To sum up, we have

\vspace{-0.2cm}

\begin{equation}
\nu_{\infty,\lambda}^p(\textnormal{MC}_{\lambda_0}(M))=\left\{
\begin{array}{cl}
  \nu_{\infty,\lambda\lambda_0}^{p-1}(M) \ &\textnormal{if } \gamma \in (0,1-\gamma_0) \\[1mm]
  \nu_{\infty,\lambda\lambda_0}^{p}(M) \ &\textnormal{if } \gamma \in (1-\gamma_0,1).
\end{array}
\right.
\end{equation}

\vspace{0.4cm}

\noindent
Let us now take the monodromy filtration $M_\bullet$ into account. We know that it induces a filtration $W_\bullet$ on the cohomology $H^\bullet (K^\bullet)$ defined by $W_\ell H^m(K^\bullet)=\textnormal{Im}(H^m(M_\ell(K^\bullet)) \rightarrow H^m(K^\bullet))$ which verifies $\chi(\textnormal{gr}_\ell^M K^\bullet)=\chi(\textnormal{gr}_\ell^W H^\bullet(K^\bullet))$, by a classical argument of degeneration of spectral sequence (see \cite[Lemma 4.2.6]{Mar18} for more details). Lemma \ref{hj} says that $\textnormal{gr}_\ell^M K^\bullet$ is concentrated in degree $0$, so that the previous equality of Euler characteristics gives $\dim H^0(\textnormal{gr}_\ell^M K^\bullet)=\dim(\textnormal{gr}_\ell^W H^0(K^\bullet))=\dim \textnormal{gr}_\ell^W \psi_{\infty,\lambda}(\textnormal{MC}_{\lambda_0}(M))$. It remains to observe that $W_\bullet$ coincides here with $M_\bullet$, because $W_\bullet$ verifies the two properties which characterize the monodromy filtration \cite[Lemma 3.1.1]{SS19}. The same reasoning holds with primitive parts instead of graded parts, hence we have the equality

\vspace{-0.1cm}

$$\dim \prim_\ell \psi_{\infty,\lambda}(\textnormal{MC}_{\lambda_0}(M))=\dim \bm{H}^1(\mathbb{P}^1,\DR \prim_\ell T^\lambda).$$

\vspace{0.4cm}

\noindent
Now, let us fix $\ell \geq 0$ and consider the Hodge filtration of the complex $\prim_\ell K^\bullet$. \hspace{-2.3mm} As the connection sends the $p$-th piece of the filtration to the $(p-1)$-th piece of the filtration, we have $H^j(\textnormal{gr}_{F}^{p} \prim_\ell K^\bullet)=0$ for $j \neq 0$ and $p \in \mathbb{Z}$ similarly to Lemma \ref{hj}. Using the previous argument again, with the Hodge filtration this time, we have

\vspace{-0.2cm}

$$\nu_{\infty,\lambda,\ell}^p(\textnormal{MC}_{\lambda_0}(M))=\dim  \textnormal{gr}_F^p \prim_\ell \psi_{\infty,\lambda}(\textnormal{MC}_{\lambda_0}(M))=\dim \textnormal{gr}_F^p \bm{H}^1(\mathbb{P}^1,\DR \prim_\ell T^\lambda).$$

\vspace{0.3cm}

\noindent
As $\DR \prim_\ell T^\lambda$ is again of the form $i_* \mathscr{V}$, we can apply the same reasoning as for $T^\lambda$, and get a formula similar to (\ref{eq1}) for $\prim_\ell T^\lambda$:

\vspace{-0.2cm}

\begin{equation}\label{eq1bis}
\nu_{\infty,\lambda,\ell}^p(\textnormal{MC}_{\lambda_0}(M))=\delta^{p-1}(\prim_\ell T^\lambda)-\delta^p(\prim_\ell T^\lambda) - h^p(\prim_\ell T^\lambda) +2 h^{p-1}(\prim_\ell T^\lambda).
\end{equation}

\vspace{0.3cm}

\noindent
On the one hand, we have

\vspace{-0.2cm}

$$\prim_\ell T^\lambda=\prim_\ell (M')^{\lambda \lambda_0} \otimes \left( \C[x',x'^{-1},(x'-1)^{-1}], \textnormal{d}+\gamma_0 \left(-\frac{\textnormal{d}x'}{x'}+\frac{\textnormal{d}x'}{x'-1} \right) \right),$$

\vspace{0.1cm}

\noindent
and on the other hand, we have a formula similar to (\ref{eq2}) with $\prim_\ell (M')^{\lambda \lambda_0}$:

\vspace{-0.2cm}

$$0=\delta^{p-1}(\prim_\ell(M')^{\lambda \lambda_0})-\delta^p(\prim_\ell(M')^{\lambda \lambda_0}) - h^p(\prim_\ell(M')^{\lambda \lambda_0}) + h^{p-1}(\prim_\ell(M')^{\lambda \lambda_0}).$$

\vspace{0.3cm}

\noindent
So we can repeat the reasoning as without the monodromy filtration and get

\vspace{-0.1cm}

$$\nu_{\infty,\lambda,\ell}^p(\textnormal{MC}_{\lambda_0}(M))=\left\{
\begin{array}{cl}
  \nu_{\infty,\lambda\lambda_0,\ell}^{p-1}(M) \ &\textnormal{if } \gamma \in (0,1-\gamma_0) \\[1mm]
  \nu_{\infty,\lambda\lambda_0,\ell}^{p}(M) \ &\textnormal{if } \gamma \in (1-\gamma_0,1).
\end{array}
\right.
$$

\vspace{0.4cm}

\noindent
(Case 2) Let us look at the case $\lambda=1$, which differs from the previous one locally at the neighbourhood of $0$ where we have a minimal extension.

\vspace{0.2cm}

\noindent
We claim that, for every $\ell\geq0$, the primitive Hodge module $\prim_\ell T^1$ decomposes as $\prim^1_\ell T^1\oplus\prim^2_\ell T^1$, where $\prim^2_\ell T^1$ is supported at the origin and $\prim^1_\ell T^1$ is smooth near the origin. Since the question is local at the origin, it is enough to consider the quiver associated to $T^1$ at the origin. By the minimal extension property, it takes the form
\[
\xymatrix{
H \quad \ar@<2pt>[r] & \quad \mathrm{N}(H) \ar@<2pt>[l]
}
\]
where $H$ denotes the nearby cycles of $T^1$ at the origin. As a filtered vector space, it is identified which $(M')_{\gamma_0}$ according to Remark \ref{remarque}.

\vspace{0.3cm}

\noindent
Let us set $G=\mathrm{N}(H)$. The action of the nilpotent operator on $T^1$ is induced by that of $\mathrm{N}$ on $H$ and $G$. Since $\mathrm{N}$ sends $M_\ell H$ to $M_{\ell-1}G$, we deduce that the local quiver of $\prim_\ell T^1$ is
\[
\xymatrix{
\prim_\ell H \quad \ar@<2pt>[r]^-{0} & \quad \prim_\ell G. \ar@<2pt>[l]^-{0}
}
\]

\vspace{0.2cm}

\noindent
We define $\mathrm{P}^2_\ell T^1$ by the quiver
\[
\xymatrix{
0 \quad \ar@<2pt>[r]^-{0} & \quad \prim_\ell G \ar@<2pt>[l]^-{0}
}
\]
and $\mathrm{P}^1_\ell T^1$ is the smooth extension of $\mathrm{P}_\ell T^1$ at the origin, locally defined by the quiver
\[
\xymatrix{
\prim_\ell H \quad \ar@<2pt>[r]^-{0} & \quad 0 \ar@<2pt>[l]^-{0}
}
\]
proving thereby the claim.

\vspace{0.3cm}

\noindent
By its definition, this decomposition underlies a decomposition of pure Hodge modules, so that we have

\vspace{-0.2cm}

\begin{align*}
\nu_{\infty,1,\ell}^p(\textnormal{MC}_{\lambda_0}(M)) &= \dim \textnormal{gr}_F^p \bm{H}^1(\mathbb{P}^1,\DR(\prim_\ell^1 T^1 \oplus \prim_\ell^2 T^1)) \\
&= \dim \textnormal{gr}_F^p \bm{H}^1(\mathbb{P}^1,\DR \prim_\ell^1 T^1)+\dim \textnormal{gr}_F^p \bm{H}^1(\mathbb{P}^1,\DR \prim_\ell^2 T^1),
\end{align*}

\vspace{0.15cm}

\noindent
and, therefore, two dimensions to calculate. The first one can be obtained by considering that $\DR \prim_\ell^1 T^1$ is again of the form $i_* \mathscr{V}$ where $\mathscr{V}$ is a non-constant local system on $\mathbb{P}_{\textnormal{exc}}^1 \setminus \bm{x}$, but with only two singular points ($0$ and $1)$. For each of them, the unique eigenvalue of the monodromy is different from $1$. This gives $H^0(\mathbb{P}^1_{\textnormal{exc}},i_* \mathscr{V})=0$ and $H^2(\mathbb{P}^1_{\textnormal{exc}},i_* \mathscr{V})=0$ by the Poincar\'e duality theorem. Consequently we have

\vspace{-0.2cm}

$$\dim H^1(\mathbb{P}^1_{\textnormal{exc}},i_* \mathscr{V})=-\chi(\mathbb{P}^1_{\textnormal{exc}},i_* \mathscr{V})=-\underbrace{\chi(\mathbb{P}^1_{\textnormal{exc}}\setminus\{0,1\})}_{=0} \rk(\mathscr{V})=0.$$

\vspace{0.05cm}

\noindent
Finally:

\vspace{-0.1cm}

$$\dim \textnormal{gr}_F^p \bm{H}^1(\mathbb{P}^1,\DR \prim_\ell^1 T^1) = 0.$$

\vspace{0.2cm}

\noindent
Let us now try to determine $\dim \textnormal{gr}_F^p \bm{H}^1(\mathbb{P}^1,\DR \prim_\ell^2 T^1)$. We know that $\prim_\ell^2 T^1$ is supported at 0 and given by $(\prim_\ell G)[\partial_{x'}]$. Applying Remark \ref{remarque}, the Hodge filtration is given by

$$F^p((\prim_\ell G)[\partial_{x'}])=\sum_{k \geq 0} \: \partial_{x'}^{k} \cdot F^{p+1+k}(\prim_\ell G).$$

\vspace{0.25cm}

\noindent
We deduce that $\bm{H}^1(\mathbb{P}^1,\DR \prim_\ell^2 T^1)$ is given by the cokernel of

$$(\prim_\ell G)[\partial_{x'}] \overset{\partial_{x'}}{\longrightarrow} (\prim_\ell G)[\partial_{x'}],$$

\vspace{0.35cm}

\noindent
which can be identified with $\prim_\ell G$ equipped with the filtration

$$F^p \bm{H}^1(\mathbb{P}^1,\DR \prim_\ell^2 T^1)=F^p (\prim_\ell G).$$

\vspace{0.4cm}

\noindent
Finally, we have

\vspace{-0.2cm}

\begin{align*}
\dim \textnormal{gr}_F^p \bm{H}^1(\mathbb{P}^1,\DR \prim_\ell^2 T^1) &= \dim \textnormal{gr}_F^p (\prim_\ell G) \\
&= \dim \textnormal{gr}_F^p (\prim_{\ell+1} H) \\
&= \nu_{\infty,\lambda_0,\ell+1}^p(M).
\end{align*}

\vspace{0.25cm}

\noindent
Summing the two dimensions, we get

\vspace{-0.1cm}

$$\nu_{\infty,1,\ell}^p(\textnormal{MC}_{\lambda_0}(M))=0+\nu_{\infty,\lambda_0,\ell+1}^p(M)=\nu_{\infty,\lambda_0,\ell+1}^p(M).$$

\vspace{0.6cm}

\noindent
(Case 3) If $\lambda=\overline{\lambda_0}$, we take again the exact sequence $0 \rightarrow T_0 \rightarrow T^\lambda \rightarrow T_1 \rightarrow 0$ and we have

\vspace{-0.1cm}

$$\nu_{\infty,\bar{\lambda_0}}^p(\textnormal{MC}_{\lambda_0}(M))= \dim \textnormal{gr}_F^p \bm{H}^1(\mathbb{P}_{\textnormal{exc}}^1,\DR T_0)+\dim \textnormal{gr}_F^p \bm{H}^1(\mathbb{P}_x^1,\DR T_1).$$

\vspace{0.4cm}

\noindent
The case of the left term can be treated in the same way as for $\gamma \in (1-\gamma_0,1)$ in Case 1, that gives $\dim \textnormal{gr}_F^p \bm{H}^1(\mathbb{P}_{\textnormal{exc}}^1,\DR T_0)=\nu_{\infty,1}^{p}(M)$. Concerning the right term, we have

\vspace{-0.1cm}

$$\dim \textnormal{gr}_F^p \bm{H}^1(\mathbb{P}_x^1,\DR T_1)=\dim \textnormal{gr}_F^p \bm{H}^1(\mathbb{P}^1,\DR\mathscr{M})=h^p H^1(\A^1,\DR M),$$

\vspace{0.35cm}

\noindent
where $\mathscr{M}$ is the meromorphic extension of $M$ at infinity. By \cite[Lemma 2.2.8 and Remark 2.3.5]{DS13}, we have

\vspace{-0.1cm}

\begin{equation}\label{eq5}
h^p H^1(\A^1,\DR M)=h^p H^1(\mathbb{P}^1,\DR\mathscr{M}^{\textnormal{min}})+\nu_{\infty,1,\textnormal{prim}}^{p-1}(M),
\end{equation}

\vspace{0.2cm}

\noindent
and we get

\vspace{-0.3cm}

\begin{equation}\label{eq6}
\nu_{\infty,\bar{\lambda_0}}^p(\textnormal{MC}_{\lambda_0}(M)) = \nu_{\infty,1}^{p}(M) + h^p H^1(\mathbb{P}^1,\DR\mathscr{M}^{\textnormal{min}}) + \nu_{\infty,1,\textnormal{prim}}^{p-1}(M).
\end{equation}

\vspace{0.3cm}

\noindent
Now, we have seen in the proof of Lemma \ref{hj} that we are in a situation of a minimal extension quiver

\vspace{-0.3cm}

$$\xymatrix{
T^\lambda \quad \ar@<4pt>@{->>}[r]^-{\mathrm{N}} & \quad T_0 \ar@<4pt>@{_{(}->}[l]
},
$$

\vspace{0.1cm}

\noindent
with $\prim_\ell T^\lambda \simeq \prim_{\ell-1} T_0$ for $\ell \geq 1$. As $\mathrm{N}$ is strictly compatible with the Hodge filtration, with a shift $F^\bullet \rightarrow F^{\bullet-1}$, we deduce that $\textnormal{gr}_F^p \prim_\ell T^\lambda \simeq \textnormal{gr}_F^{p-1} \prim_{\ell-1} T_0$ for $\ell \geq 1$, and so

\vspace{-0.2cm}

$$\nu_{\infty,\bar{\lambda_0},\ell}^p(\textnormal{MC}_{\lambda_0}(M)) = \dim \textnormal{gr}_F^{p-1} \bm{H}^1(\mathbb{P}_{\textnormal{exc}}^1,\DR \prim_{\ell-1} T_0)=\nu_{\infty,1,\ell-1}^{p-1}(M).$$

\vspace{0.4cm}

\noindent
It remains to treat the case $\ell=0$, for which we have
$$\nu_{\infty,\bar{\lambda_0}}^p(\textnormal{MC}_{\lambda_0}(M))=\nu_{\infty,\bar{\lambda_0},0}^p(\textnormal{MC}_{\lambda_0}(M))+\sum_{\ell \geq 1} \sum_{k=0}^\ell \nu_{\infty,\bar{\lambda_0},\ell}^{p+k}(\textnormal{MC}_{\lambda_0}(M))$$

\noindent
and

\vspace{-0.5cm}

\begin{align*}
\sum_{\ell \geq 1}\sum_{k=0}^\ell \nu_{\infty,\bar{\lambda_0},\ell}^{p+k}(\textnormal{MC}_{\lambda_0}(M)) &= \sum_{\ell \geq 1} \sum_{k=0}^\ell \nu_{\infty,1,\ell-1}^{p-1+k}(M) \\
&= \sum_{\ell \geq 0} \sum_{k=0}^{\ell} \nu_{\infty,1,\ell}^{p-1+k}(M) + \sum_{\ell \geq 0} \nu_{\infty,1,\ell}^{p+\ell}(M) \\
&= \nu_{\infty,1}^{p-1}(M)+\nu_{\infty,1,\textnormal{coprim}}^{p}(M),
\end{align*}

\vspace{0.3cm}

\noindent
so we deduce that

\vspace{-0.2cm}

\begin{multline*}
\nu_{\infty,\bar{\lambda_0},0}^p(\textnormal{MC}_{\lambda_0}(M))=h^p H^1(\mathbb{P}^1,\DR\mathscr{M}^{\textnormal{min}}) + \nu_{\infty,1,\textnormal{prim}}^{p-1}(M) - \nu_{\infty,1,\textnormal{coprim}}^{p}(M)\\
+ \nu_{\infty,1}^{p}(M) - \nu_{\infty,1}^{p-1}(M).
\end{multline*}

\vspace{0.4cm}

\noindent
Yet, a general calculation \cite[(2.2.5$*$)]{DS13} immediately shows that

\vspace{-0.2cm}

$$\nu_{\infty,1,\textnormal{coprim}}^{p}(M) - \nu_{\infty,1,\textnormal{prim}}^{p-1}(M)=\nu_{\infty,1}^{p}(M) - \nu_{\infty,1}^{p-1}(M),$$

\vspace{0.3cm}

\noindent
and we conclude that $\nu_{\infty,\bar{\lambda_0},0}^p(\textnormal{MC}_{\lambda_0}(M))=h^p H^1(\mathbb{P}^1,\DR\mathscr{M}^{\textnormal{min}})$, which ends the proof of Theorem \ref{theo1}.
\end{proof}

\vspace{0.1cm}

\section{Proof of Theorems \ref{theo2} and \ref{theo3}}

\vspace{0.5cm}

\begin{proof}[Proof of Theorem \ref{theo2}]
Applying identity (2.2.2$**$) of \cite{DS13}, we have

$$h^p(\textnormal{MC}_{\lambda_0}(M))=\sum_{\lambda \in S^1} \nu_{\infty,\lambda}^p(\textnormal{MC}_{\lambda_0}(M)),$$

\vspace{0.3cm}

\noindent
so it suffices to sum expressions from Theorem \ref{theo1} (recall that $\lambda=\exp(-2i\pi\gamma$)):

\begin{align}
\sum_{\lambda \neq 1, \bar{\lambda_0}} \nu_{\infty,\lambda}^p(\textnormal{MC}_{\lambda_0}(M))=\sum_{\gamma \in (0,\gamma_0)} \nu_{\infty,\lambda}^{p}(M) + \sum_{\gamma \in (\gamma_0,1)} \nu_{\infty,\lambda}^{p-1}(M)
\end{align}

\vspace{-0.2cm}

\begin{align}
\nu_{\infty,1}^p(\textnormal{MC}_{\lambda_0}(M)) &= \sum_{\ell \geq 0}\sum_{k=0}^\ell \nu_{\infty,\lambda_0,\ell+1}^{p+k}(M) \\
&= \nu_{\infty,\lambda_0}^{p}(M)-\nu_{\infty,\lambda_0,\textnormal{coprim}}^{p}(M) \notag \\
&= \nu_{\infty,\lambda_0}^{p-1}(M)-\nu_{\infty,\lambda_0,\textnormal{prim}}^{p-1}(M) \notag
\end{align}

\vspace{0.3cm}

\noindent
For $\lambda=\bar{\lambda_0}$, the following equality has already been proved in {\S}\ref{section2}, formulas \eqref{eq5} and \eqref{eq6}:

\vspace{-0.2cm}

\begin{equation}
\nu_{\infty,\bar{\lambda_0}}^p(\textnormal{MC}_{\lambda_0}(M)) = \nu_{\infty,1}^{p}(M) + h^p H^1(\mathbb{A}^1,\DR M)
\end{equation}

\vspace{-0.2cm}

\end{proof}

\vspace{0.2cm}

\begin{proof}[Proof of Theorem \ref{theo3}]
We set $\gamma^p=\delta^p-\delta^{p-1}$. According to identity (2.3.5$*$) of \cite{DS13}, we have

\vspace{-0.2cm}

$$h^p H^1(\mathbb{A}^1,\DR M) = -\gamma^p(M) - h^p(M) + \sum_{i=1}^r \Bigg( \sum_{\mu \neq 1} \mu_{x_i,\mu}^{p-1}(M) + \mu_{x_i,1}^p(M) \Bigg).$$

\vspace{0.2cm}

\noindent
It then follows from Theorem \ref{theo2} that

\vspace{-0.2cm}

\begin{multline}\label{eqdeltap}
h^p(\textnormal{MC}_{\lambda_0}(M)) + h^p(M) = -\gamma^p(M) - \nu_{\infty,\lambda_0,\textnormal{prim}}^{p-1}(M) + \sum_{\gamma \in [0,\gamma_0)} \nu_{\infty,\lambda}^{p}(M) \\
+ \sum_{\gamma \in [\gamma_0,1)} \nu_{\infty,\lambda}^{p-1}(M) + \sum_{i=1}^r \Bigg( \sum_{\mu \neq 1} \mu_{x_i,\mu}^{p-1}(M) + \mu_{x_i,1}^p(M) \Bigg).
\end{multline}

\vspace{0.3cm}

\noindent
According to \cite[Prop. 3.1.1]{DS13}, we have
$$h^p(\textnormal{MC}_{\bar{\lambda_0}}(\textnormal{MC}_{\lambda_0}(M)))=h^{p-1}(M).$$

\noindent
We can write the same formula as above with $\bar{\lambda_0}$ instead of $\lambda_0$, then we apply it to $\textnormal{MC}_{\lambda_0}(M)$ instead of $M$:

\vspace{-0.2cm}

\begin{multline}\label{eq35}
h^p(\textnormal{MC}_{\lambda_0}(M)) +h^{p-1}(M) = -\gamma^p(\textnormal{MC}_{\lambda_0}(M)) - \nu_{\infty,\bar{\lambda_0},\textnormal{prim}}^{p-1}(\textnormal{MC}_{\lambda_0}(M)) \\[2mm]
+ \sum_{\gamma \in [0,1-\gamma_0)} \nu_{\infty,\lambda}^{p}(\textnormal{MC}_{\lambda_0}(M)) + \sum_{\gamma \in [1-\gamma_0,1)} \nu_{\infty,\lambda}^{p-1}(\textnormal{MC}_{\lambda_0}(M)) \\[2mm]
+ \sum_{i=1}^r \Bigg( \sum_{\mu \neq 1} \mu_{x_i,\mu}^{p-1}(\textnormal{MC}_{\lambda_0}(M)) + \mu_{x_i,1}^p(\textnormal{MC}_{\lambda_0}(M)) \Bigg).
\end{multline}

\vspace{0.3cm}

\noindent
It follows from Theorem \ref{theo1} that

\begin{align}\label{eq36}
\sum_{\gamma \in [0,1-\gamma_0)} \nu_{\infty,\lambda}^{p}(\textnormal{MC}_{\lambda_0}(M)) = \sum_{\gamma \in (\gamma_0,1)} \nu_{\infty,\lambda}^{p-1}(M) + \nu_{\infty,1}^{p}(\textnormal{MC}_{\lambda_0}(M)),
\end{align}

\begin{align}\label{eq37}
\sum_{\gamma \in [1-\gamma_0,1)} \nu_{\infty,\lambda}^{p-1}(\textnormal{MC}_{\lambda_0}(M)) = \sum_{\gamma \in (0,\gamma_0)} \nu_{\infty,\lambda}^{p-1}(M) + \nu_{\infty,\bar{\lambda_0}}^{p-1}(\textnormal{MC}_{\lambda_0}(M)).
\end{align}

\vspace{0.4cm}

\noindent
Moreover, according to \cite[Theorem 3.1.2(2)]{DS13}, we have

\vspace{-0.2cm}

\begin{multline}\label{eq38}
\sum_{i=1}^r \Bigg( \sum_{\mu \neq 1} \mu_{x_i,\mu}^{p-1}(\textnormal{MC}_{\lambda_0}(M)) + \mu_{x_i,1}^p(\textnormal{MC}_{\lambda_0}(M)) \Bigg) =\\
\sum_{i=1}^r \Bigg( \sum_{\gamma \in (0,1-\gamma_0)} \mu_{x_i,\lambda}^{p-2}(M) + \sum_{\gamma \in [1-\gamma_0,1]} \mu_{x_i,\lambda}^{p-1}(M) \Bigg).
\end{multline}

\vspace{0.3cm}

\noindent
Substituting \eqref{eq36}, \eqref{eq37} and \eqref{eq38} in Formula \eqref{eq35}, we get:

\vspace{-0.2cm}

\begin{multline}
h^p(\textnormal{MC}_{\lambda_0}(M)) = -\gamma^p(\textnormal{MC}_{\lambda_0}(M)) \underbrace{- h^{p-1}(M) + \sum_{\gamma \neq 0,\gamma_0} \nu_{\infty,\lambda}^{p-1}(M)}_{=-\nu_{\infty,1}^{p-1}(M)-\nu_{\infty,\lambda_0}^{p-1}(M)} \\[2mm]
+ \underbrace{\nu_{\infty,1}^{p}(\textnormal{MC}_{\lambda_0}(M))}_{=\nu_{\infty,\lambda_0}^{p-1}(M)-\nu_{\infty,\lambda_0,\textnormal{prim}}^{p-1}(M)} + \: \nu_{\infty,\bar{\lambda_0}}^{p-1}(\textnormal{MC}_{\lambda_0}(M)) - \nu_{\infty,\bar{\lambda_0},\textnormal{prim}}^{p-1}(\textnormal{MC}_{\lambda_0}(M)) \\[2mm]
+ \sum_{i=1}^r \Bigg( \sum_{\gamma \in (0,1-\gamma_0)} \mu_{x_i,\lambda}^{p-2}(M) + \sum_{\gamma \in [1-\gamma_0,1]} \mu_{x_i,\lambda}^{p-1}(M) \Bigg).
\end{multline}

\vspace{0.3cm}

\noindent
We already made $\nu_{\infty,\bar{\lambda_0}}^{p}(\textnormal{MC}_{\lambda_0}(M))$ explicit in the proof of Theorem \ref{theo2}, but we can remark that

\vspace{-0.2cm}

\begin{align*}
\nu_{\infty,\bar{\lambda_0}}^{p-1}(\textnormal{MC}_{\lambda_0}(M)) - \nu_{\infty,\bar{\lambda_0},\textnormal{prim}}^{p-1}(\textnormal{MC}_{\lambda_0}(M)) &= \sum_{\ell \geq 1} \sum_{k=1}^\ell \nu_{\infty,\bar{\lambda_0},\ell}^{p-1+k}(\textnormal{MC}_{\lambda_0}(M)) \\
&= \sum_{\ell \geq 0} \sum_{k=0}^\ell \nu_{\infty,1,\ell}^{p-1+k}(M) = \nu_{\infty,1}^{p-1}(M).
\end{align*}

\vspace{0.4cm}

\noindent
Finally, we get

\vspace{-0.2cm}

\begin{multline*}
h^p(\textnormal{MC}_{\lambda_0}(M)) = -\gamma^p(\textnormal{MC}_{\lambda_0}(M))- \nu_{\infty,\lambda_0,\textnormal{prim}}^{p-1}(M) \\
+ \sum_{i=1}^r \Bigg( \sum_{\gamma \in (0,1-\gamma_0)} \mu_{x_i,\lambda}^{p-2}(M) + \sum_{\gamma \in [1-\gamma_0,1]} \mu_{x_i,\lambda}^{p-1}(M) \Bigg).
\end{multline*}

\vspace{0.4cm}

\noindent
If we substitute \eqref{eqdeltap} in the previous formula, we have

\vspace{-0.2cm}

\begin{multline*}
\gamma^p(\textnormal{MC}_{\lambda_0}(M)) = \gamma^p(M) + \sum_{\gamma \in [\gamma_0,1)} (\nu_{\infty,\lambda}^{p}(M)-\nu_{\infty,\lambda}^{p-1}(M)) \\
- \sum_{i=1}^r \bigg( (\mu_{x_i,1}^p(M)-\mu_{x_i,1}^{p-1}(M)) + \sum_{\gamma \in (0,1-\gamma_0)} (\mu_{x_i,\lambda}^{p-1}(M)-\mu_{x_i,\lambda}^{p-2}(M)) \bigg).
\end{multline*}

\vspace{0.4cm}

\noindent
Summing these equalities for $p' \leq p$ gives the announced formula.
\end{proof}

\vspace{0.3cm}

\begin{rema}\label{rema}
If we add the assumption that the scalar monodromy at infinity is in fact equal to $\lambda_0 \textnormal{Id}$ as in \cite{DS13}, we have $\nu_{\infty,\lambda,\ell}^{p}(M)=0$ except if $\lambda=\lambda_0$ and $\ell=0$. Thus we have

\vspace{-0.2cm}

$$\sum_{\gamma \in [0,\gamma_0)} \nu_{\infty,\lambda}^{p}(M) + \sum_{\gamma \in [\gamma_0,1)} \nu_{\infty,\lambda}^{p-1}(M) =\nu_{\infty,\lambda_0}^{p-1}(M)=\nu_{\infty,\lambda_0,\textnormal{prim}}^{p-1}(M)=h^{p-1}(M)$$

\vspace{0.3cm}

\noindent
and

\vspace{-0.2cm}

$$\sum_{\gamma \in [\gamma_0,1)} \nu_{\infty,\lambda}^{p}(M)=h^p(M),$$

\vspace{0.3cm}

\noindent
and consequently we retrieve the results 3.1.2(1) and 3.1.2(3) of \cite{DS13}.

\end{rema}

\section*{Acknowledgements}

We thank first of all Claude Sabbah to whom this work owes a lot. The author is indebted as well to Michel Granger and Christian Sevenheck for their careful reading and constructive comments about this work. We also thank Michael Dettweiler for helpful discussions.

\bibliographystyle{amsalpha}
\bibliography{bibli}

\backmatter

\end{document}